\documentclass[dvipsnames]{amsart}
\usepackage[utf8]{inputenc}
\usepackage{graphicx,tikz-cd}
\usepackage{amssymb}
\usepackage{amsmath,tensor}
\usepackage{float}
\usepackage{xcolor}
\usepackage[
style=alphabetic,
maxbibnames=9,
doi=false,isbn=false,url=false,
]{biblatex}
\usepackage{mathscinet}

\usepackage{hyperref}
\usepackage{cleveref}
\usepackage{comment}
\usepackage[all]{xy}

\newtheorem{thm}{Theorem}[section] 
\newtheorem{prop}[thm]{Proposition}
\newtheorem{lem}[thm]{Lemma}
\newtheorem{cor}[thm]{Corollary}

\theoremstyle{definition}
\newtheorem{defn}[thm]{Definition}
\newtheorem{notation}[thm]{Notation}
\newtheorem{ex}[thm]{Example}

\theoremstyle{remark}
\newtheorem{rem}[thm]{Remark}

\newcommand{\uw}{\underline{w}}
\newcommand{\ux}{\underline{x}}
\newcommand{\uy}{\underline{y}}
\newcommand{\uz}{\underline{z}}

\def\Z{\mathbb{Z}}
\newcommand{\gdim}{\mathrm{gd}}

\newcommand{\ot}{\otimes}
\newcommand{\co}{\colon}

\DeclareMathOperator{\Hom}{Hom}
\DeclareMathOperator{\End}{End}

\DeclareMathOperator{\id}{id}

\newcommand{\im}{\operatorname{Im}}
\newcommand{\inv}{^{-1}}

\newcommand{\lei}{_i}

\newcommand{\mi}{\underline}
\newcommand{\ma}{\overline}
\newcommand{\expr}{\leftrightharpoons}

\newcommand{\SSBim}{\mathcal{S}\mathbb{S}\textrm{Bim}}

\newcommand{\mt}{\emptyset}

\newcommand{\ch}[1]{\mathrm{ch}{(#1)}}
\newcommand{\onetensor}{1^\otimes}
\newcommand{\Bim}{\mathrm{Bim}}
\newcommand{\SBim}{\mathbb{S}\textrm{Bim}}
\newcommand{\BSBim}{\mathbb{BS}\textrm{Bim}}
\newcommand{\SBSBim}{\mathcal{S}\mathbb{BS}\textrm{Bim}}

\newcommand{\res}{\operatorname{Res}}
\newcommand{\ind}{\operatorname{Ind
}}
\newcommand{\proj}{\operatorname{proj}}
\newcommand{\incl}{\operatorname{inj}}

\DeclareMathOperator{\BS}{BS}

\DeclareMathOperator{\leftred}{LR}
\DeclareMathOperator{\rightred}{RR}

\DeclareMathOperator{\term}{term}
\DeclareMathOperator{\Term}{Term}
\newcommand{\lep}{{\{\leq p\}}}

\title{Subexpressions and the Bruhat order for double cosets}
\author{Ben Elias, Hankyung Ko, Nicolas Libedinsky, Leonardo Patimo}
\begin{document}
\maketitle

\begin{abstract}

The Bruhat order on a Coxeter group is often described by examining subexpressions of a reduced expression. We prove that an analogous description applies to the Bruhat order on double cosets. This establishes the compatibility of the Bruhat order on double cosets with concatenation, leading to compatibility between the monoidal structure and the ideal of lower terms in the singular Hecke 2-category. We also prove other fundamental properties of this ideal of lower terms.


\end{abstract}

\section{Introduction}

Let $(W,S)$ be a Coxeter system. There are many equivalent definitions of the Bruhat order on $W$. One of the most useful definitions is this: $x \le y$ if a reduced expression for $x$ appears as a subexpression of a (or equivalently any) reduced expression for $y$.

In \cite{SingSb}, Williamson defines the concept of a reduced expression for a double coset $p \in W_I \backslash W / W_J$. Here $I$ and $J$ are subsets of $S$, and $W_I$ and $W_J$ the parabolic subgroups they generate, which are assumed to be finite. The recent paper \cite{EKo} expands greatly on this concept, defining expressions (not necessarily reduced) and giving many equivalent and more practical definitions of reduced expressions. 
One important omission from \cite{EKo}, which we rectify in this paper, is a definition of the Bruhat order in terms of subexpressions of reduced expressions.

Expressions for double cosets have a different flavor than ordinary expressions. Ordinary expressions  are lists in $S$.  One can omit some of the elements of this list and their multiplication is still a well-defined element of $W$. Meanwhile,   expressions for double cosets are lists of subsets of $S$, where two subsets adjacent in the list  differ by one element of $S$. One can not omit some of the elements in such a list.

The appropriate replacement for subexpressions is the notion of a \emph{path subordinate to an expression} (see Definition~\ref{defn:pathsubord}). 
A subordinate path is a sequence of double cosets obtained from an expression, but using a nondeterministic version of multiplication. This is analogous to the non-determinism in an ordinary subexpression, where instead of multiplying by a simple reflection $s$, we are permitted to multiply by either $1$ or $s$. The appropriate replacement for ``the element of $W$ associated to the subexpression''  is the \emph{terminus} of a path.

The Bruhat order on cosets is usually defined by $p \le q$ if and only if $\mi{p} \le \mi{q}$ in the ordinary Bruhat order. Here $\mi{p}$ and $\mi{q}$ are the minimal elements in these cosets. Our main result, Theorem~\ref{thm:bruhat}, states that $p \le q$ if and only if $p$ is the terminus of a path subordinate to a (or equivalently any) reduced expression for $q$.

One important consequence of the subexpression definition of Bruhat order is its compatibility with concatenation. Suppose that the product $xyz$ is reduced, in that $\ell(xyz) = \ell(x) + \ell(y) + \ell(z)$, and that $y' < y$. Then $xy' z < xyz$.

An even stronger statement is true. Note that $x y' z$ might not be reduced. Then any element obtained as a subexpression of the concatenation of reduced expressions of $x$, $y'$, and $z$ is also strictly less than $xyz$.

\begin{ex} 
Consider $x=z=s$ and $ t=y$ for $s$ and $t$ non-commuting simple reflections, and $y' = 1 \le  y$. The first statement says that $1 = ss < sts$. The stronger statement says that any subexpression of $(s,s)$ expresses an element smaller than $sts$. This implies that $s < sts$. \end{ex}

One can phrase the stronger result succinctly using the $*$-product (or Demazure product)  rather than the ordinary product in $W$. To whit,  $x * y' * z < x * y * z$. As a consequence of our main theorem, we prove  in Theorem~\ref{thm:bruhatconcatcompat} the corresponding result for double cosets.

One of the reasons that the compatibility of Bruhat order with concatenation is important is that it plays a role in the Hecke category (e.g. the category of Soergel bimodules), where concatenation is lifted to a monoidal structure. Meanwhile, expressions for double cosets yield  1-morphisms
in the singular Hecke 2-category (e.g. the 2-category of singular Soergel bimodules \cite{SingSb}). By virtue of Theorem~\ref{thm:bruhatconcatcompat}, we establish in Proposition~\ref{prop:lowertermslocal} a compatibility between the monoidal structure and certain ideals of lower terms in the singular Hecke 2-category, which we define in this paper.


One of the reasons that the compatibility of the Bruhat order on $W$ with concatenation is important is that it plays a role in the Hecke category (e.g. the category of Soergel bimodules), where concatenation is lifted to a monoidal structure. Let us explain at a high level, without dwelling on the details (which can be found in \cite{GBM}), and then discuss the double coset analogue.

One can equip all morphism spaces $\Hom(B,B')$ between two Soergel bimodules with a filtration by ideals, indexed by the Bruhat order on $W$. That is, for any downward-closed subset $C$ of the Bruhat order on $W$ (such as the set $\{\le w\} = \{x \in W \mid x \le w\}$), one has a subspace $\Hom_{C}(B,B') \subset \Hom(B,B')$, and the subspaces $\Hom_{C}$ are closed under pre- and post-composition. Note that an expression $\uw = (s_1, \ldots, s_d)$ (with $s_i \in S$) gives rise to a bimodule $\BS(\uw)$ in the Hecke category. One can define the ideals $\Hom_{\le w}$ in two distinct ways which are not obviously related:
\begin{itemize}
\item by examining the action of morphisms on the support filtration of a Soergel bimodule,
\item by rewriting the morphism as a linear combination of morphisms which factor through $\BS(\ux)$, where $\ux$ ranges over reduced expressions for elements $x \le w$. \end{itemize}
The first criterion is the most classical, and relates to a filtration on the bimodules themselves (rather than a filtration on Hom spaces). It relates to the characters of bimodules, which determine their images in the Grothendieck group. The second criterion is the most combinatorial. It is also intrinsic to the category itself, so it also applies to other versions of the Hecke category (e.g. geometric, diagrammatic) where bimodules themselves are absent.

There is also an extremely helpful computational tool which gives a third criterion for when a morphism is in the ideal $\Hom_{< w}$, which applies specifically to morphisms between bimodules $\BS(\uw)$ associated to reduced expressions for $w$. One need only check whether a particular minimal degree element known as the \emph{one-tensor} is in the kernel of the map, or whether the image of the map meets the (polynomial) span of the one-tensor.

Using the second criterion and the compatibility of the Bruhat order with concatenation, one can prove that whenever $xyz$ is reduced, $\BS(\ux) \otimes \Hom_{< y} \otimes \BS(\uz) \subset \Hom_{< xyz}$. Here $\BS(\ux)$ and $\BS(\uz)$ are objects associated to reduced expressions for $x$ and $z$ respectively. We think of this as a kind of compatibility between the monoidal structure and the ideal filtration.

In this paper, we prove that the singular Hecke 2-category (e.g. the 2-category of singular Soergel bimodules \cite{SingSb}) admits a filtration by ideals, indexed by the Bruhat order on double cosets. We prove that it can be accessed by  three criteria just as above. While the support filtration on singular Soergel bimodules themselves was studied by Williamson \cite{SingSb}, the filtration on Hom spaces was not. Using the second criterion, we can prove a similar compatibility between the monoidal structure and the ideal filtration (\Cref{prop:lowertermslocal}). These ideals and the various ways to study them will be crucial in future work which provides a basis for morphisms between singular Soergel bimodules.

{\bf Acknowledgments.} NL was partially supported by FONDECYT-ANID grant 1230247. BE was partially supported by NSF grant DMS-2201387.

\section{Bruhat order}\label{ss:bruhat}

\subsection{Recollections} 
 We recall some notations and results from \cite{EKo}, where the reader can find many more details and explanations.

 Let $(W,S)$ be a Coxeter system with length function $\ell$. For $I\subset S$, we denote by $W_I$ the subgroup of
$W$ generated by $I$.  When $W_I$ is finite, we say that $I$ is \textit{finitary}, and we write $w_I$ for the
longest element of $W_I$.  We write $\ell(I) := \ell(w_I)$.

For $I, J\subset S$, a \emph{$(I,J)$-coset} is an element $p$ in 
 $W_{I}\backslash W/W_{J}$. When we write ``the coset $p$" we  mean the triple $(p,I,J)$. It might happen that  $(p,I,J)\neq (p',I',J')$, even though $p=p'$ as subsets of $W$, and we distinguish between $p$ and $p'$ in this case. If $p$ is a $(I,J)$-coset we denote by $\overline{p}\in W$ and $\underline{p}\in W$ the maximal and minimal elements in the Bruhat order in the set $p$.

 A (singular) \emph{multistep  expression} is a sequence of finitary subsets of $S$ of the form 
\begin{equation}\label{ex}
 L_{\bullet}=[[ I_0\subset K_1\supset I_1\subset K_2 \supset \cdots \subset K_m\supset I_m]].
 \end{equation}
 By convention, if we write a multistep expression as $[[K_1 \supset I_1 \subset \cdots]]$, this means that $I_0 = K_1$,
and similarly $[[\cdots \subset K_m]]$ means that $I_m = K_m$. 
For $L\subset S$ and $s\in S$ we use the notation $Ls:=L\cup \{s\}$ (here we assume $s\notin L$).

A (singular)  \emph{singlestep expression} $I_{\bullet}=[I_0,I_1,\ldots, I_d]$  is a sequence of finitary subsets of $S$ such that, for all $1\leq i\leq d$, either $I_i=I_{i-1}s$  or  $I_i=I_{i-1}/ s$ for some $s\in S.$ To each singlestep expression, one can associate a multistep  expression by remembering its local maxima and minima.

\begin{defn}\label{Coxmon}
    Let $(W,S)$ be a Coxeter system. We define the \textit{Coxeter monoid} $(W,*, S)$ by the following presentation. It has generators $s\in S$ and relations 
    \begin{itemize}
\item $s*s=s$ for $s\in S.$
\item $\underbrace{s*t*s*\cdots}_{m_{st}}=\underbrace{t*s*t*\cdots}_{m_{st}}$ for all $s,t\in S. $
    \end{itemize}
\end{defn}

Elements of the Coxeter monoid $(W, *, S)$ are naturally in bijection with the elements of $W$, and we implicitly use this bijection. We write $x.y = xy$ as shorthand for the statement $xy = x * y$, which is equivalent to $\ell(x) + \ell(y) = \ell(xy)$ (see \cite[Lemma 2.3]{EKo}). If $x.y = xy$ we also say that the product $xy$ is \emph{reduced}.

\begin{lem}\label{starprod}
If $a,b,c\in W$ and $a\leq b$, then $a*c\leq b*c.$
\end{lem}
\begin{proof}
If $\ell(c)=1$, then $c$ is a simple reflection and the result follows by \cite[(2.1)]{EKo} and the lifting property \cite[Proposition 2.2.7]{BjornerBrenti}. The general case easily follows by induction on $\ell(c)$.
\end{proof}

We say that the multistep  expression $L_{\bullet}$ in  \eqref{ex}  \textit{expresses}  $p$ and we write $L_{\bullet}\expr p$ if $p$ is the unique $(I_0,I_m)$-coset with
$$
\overline{p}=w_{I_0}*w_{K_1}*w_{I_1} * (\cdots) *w_{I_m}.
$$
We say that $L_\bullet$ is \emph{reduced} if
 \begin{equation}\label{reduced}
\ma{p}=(w_{K_1}w_{I_1}^{-1}).(w_{K_2}w_{I_2}^{-1}).(\cdots).(w_{K_m}w_{I_m}^{-1})
  \end{equation}
is reduced.

A  singlestep expression $I_{\bullet}=[I_0,I_1,\ldots, I_m]$ expresses the $(I_0,I_m)$-coset of the associated multistep expression, which is also the unique coset $p$ such that
\begin{equation}\label{expressp}
\overline{p}=w_{I_0}*w_{I_1}*\cdots *w_{I_m}.
\end{equation}
We say that a singlestep expression is reduced if its associated multistep expression is reduced.

\begin{defn} (see \cite[Definition 1.2.13]{WThesis})  Let $I, J \subset S$ be finitary, and $p, q \in W_I\backslash W/W_J$. Then $p \le q$ in the Bruhat order on double cosets if and only if $\mi{p} \le \mi{q}$ in the Bruhat order on $W$. \end{defn}

\begin{notation}
We use the notation 
\[\lep=\{q\in W_I\backslash W/W_J\ |\ q\leq p\}\] for $p\in W_I\backslash W/W_J$ and, more generally, 
\[\{\leq X \}= \bigcup_{p\in X} \lep\]
for any subset $X\subset  W_I\backslash W/W_J$. A subset of the form $\{ \leq X \}$ is said to be \emph{downward-closed}.
\end{notation}

Let $L_\bullet=[[ I_0\subset K_1\supset I_1\subset K_2 \supset \cdots \subset K_m\supset I_m]]$ be a multistep expression. Recall from \cite[\S 3.4]{EKo} that the length of $L_\bullet$ is defined as 
\[\ell(L_\bullet)=-\ell(I_0)+2\ell(K_1)-2\ell(I_1)+\ldots -2\ell(I_{m-1})+2\ell(K_m)-\ell(I_m).\] 
If $p$ is a double coset, then $\ell(p):=\ell(L_\bullet)$, where $L_\bullet\expr p$ is a reduced expression. Equivalently, we have
\begin{equation} \ell(q) = 2 \ell_W(\ma{q}) - \ell(I) - \ell(J), \end{equation}
where $\ell_W$ denotes here the usual length function on $W$.
It is clear from the definition of $\ell(I_{\bullet})$ that the length of an expression will strictly increase with the addition of each index. Thus a coset $q$ has length zero if and only if $q\expr [I]$, that is, if $I=J$ and $q = q_{\id}$ is the coset containing the identity element.


\subsection{Lifting properties}

We will need the following two variations of the lifting property of the classical Bruhat order.

\begin{lem} \label{liftinglemma} Let $v, w, x \in W$ and $v \le w$. Then there exists some $x' \le x$ such that $vx \le w.x'$ (resp., $xv \le x'.w$). 
\end{lem}
\begin{proof}
First we consider the case when $x = s$ is a simple reflection. If $vs<v$ we can simply take $x'=id$, so we can assume $vs>v$.
Let $w'$ be the larger of $\{w, ws\}$ in the Bruhat order; clearly $w' = w . x'$ for some $x' \le x$. 
We wish to show that $vs \le w'$. When $ws > w$ this follows quickly from the subexpression version of the Bruhat order. When $ws < w$ this follows from the lifting property (see e.g. \cite[Proposition 2.2.7]{BjornerBrenti}).
For a general $x$, we apply the case of a simple reflection $\ell(x)$ times. \end{proof}

\begin{lem}\label{liftinglem2}
Let $v,w,x,z\in W$ be such that $v\le w$ and $w=z.x$ (resp., $w=x.z$).
Then there exists some $x'\le x^{-1}$ such that $vx'\le z$ (resp., $x'v\le z$). 
\end{lem}
\begin{proof}
First we consider the case when $x=s$ is a simple reflection, so that $w=z.s$. 
If $vs<v$, then both $v$ and $w$ have reduced expressions ending at $s$ and, by the subexpression definition of the Bruhat order, we have $vs\leq z$. So $x'=s$ works.  
If $vs>v$, then the lifting property \cite[Proposition 2.2.7]{BjornerBrenti}
gives $v\leq ws=z$ and thus $x'$ can be taken to be the identity element.
For more general $x$, we apply the case of a simple reflection $\ell(x)$ times.
\end{proof}

We state and prove the following lemma from \cite[Proposition 1.2.14]{WThesis} for convenience and completeness.
\begin{lem}\label{lem.posetprojection}
    For $I,K,J,L\subset S$ with $I\subset K$ and $J\subset L$, let
    \begin{equation}
        \pi:W_I\backslash W/W_J\to W_K\backslash W/W_{L}
    \end{equation} 
    be the quotient map. Then
    \begin{enumerate}
        \item\label{item.po}  if $p\leq p'$ in $W_I\backslash W/W_J$ then $\pi(p)\leq \pi(p')$ in $W_K\backslash W/W_L$, i.e., $\pi$ is a morphism of posets;
        \item\label{item.inversepo} for each $q\in W_K\backslash W/W_L$ we have
        \[\pi\inv (\{\leq q\})=\{ \leq \pi\inv(q)\} 
        .\]
    \end{enumerate}
\end{lem}
\begin{proof} Pick a coset $p'$ in $W_I\backslash W/W_J$. Since $\mi{p'} \in \pi(p')$ we have $x'.\mi{\pi(p')}.y'= \mi{p'}$ for some $x'\in W_K$ and $y'\in W_L$. If $p \le p'$ then $\mi{p} \le \mi{p'}$. By Lemma~\ref{liftinglem2}, there exists $x\leq (x')\inv$ and $y\leq (y')\inv$ such that $x\mi{p}y\leq \mi{\pi(p')}$. Since $x\mi{p}y\in \pi(p)$, we have $\mi{\pi(p)}\leq x\mi{p}y\leq \mi{\pi(p')}$. Thus $\pi(p) \le \pi(p')$.

Let $q\in W_K\backslash W/W_L$ and $p\in W_I\backslash W/W_J$ with $p\in \{\leq \pi\inv (q)\}$. Then there exists $p'\in \pi\inv (q)$ such that $p\leq p'$, that is, that
$\mi{p}\leq \mi{p'}= x'.\mi{q}.y'$ for some $x'\in W_K$ and $y'\in W_L$. \Cref{liftinglem2} gives $x''\in W_K$ and $y''\in W_L$ such that $x''\mi{p}y''\leq \mi{q}$. Since $x''\mi{p}y''\in \pi(p)$ we have $\mi{\pi(p)}\leq x''\mi{p}y''\leq \mi{q}$. This shows $\pi\inv (\leq q)\supset \{ \leq \pi\inv(q)\}$.

Now suppose $\pi(p)\leq q$, i.e., $\mi{\pi(p)}\leq \mi{q}$, 
and write $\mi{p}=x.\mi{\pi(p)}.y$ for $x\in W_K,y\in W_L$.
By Lemma~\ref{liftinglemma} there exist $x'\leq x, y'\leq y$ such that $\mi{p}\leq x'.\mi{q}.y'=:w$. 
Since $w\in q$, the $(I,J)$-coset $p'=W_I w W_J$ belongs to $\pi\inv(q)$.
Thus it remains to show $\mi{p}\leq \mi{p'}$. For this, we write $w=z.\mi{p'}.v$ for some $z\in W_I,v\in W_J$ and apply Lemma~\ref{liftinglem2} to $\mi{p}\leq w$. This gives $z'\in W_I, v'\in W_J$ such that $z'\mi{p}v'\leq \mi{p'}$. But $z'\mi{p}v'\in p$ so $\mi{p}\leq z'\mi{p}v'$ and thus $\mi{p} \le \mi{p'}$. 
This completes the proof of $\pi\inv (\{\leq q\})\subset \{ \leq \pi\inv(q)\}$.
\end{proof}

\subsection{Subordinate paths}

\begin{defn}\label{defn:pathsubord} (see \cite[Definition 2.17]{EKo}) Let $I_{\bullet}= [I_0, \ldots, I_d]$ be a singlestep expression. A \emph{path subordinate to $I_{\bullet}$} is a sequence $p_{\bullet} = [p_0, \ldots, p_d]$ where $p_i$ is a $(I_0, I_i)$-coset. The sequence satisfies:
\begin{itemize}
\item $p_0 = p_{\id}$, the $(I_0, I_0)$-coset containing the identity.
\item If $I_k \subset I_{k+1}$ then $p_{k+1}$ is the unique double coset containing $p_k$.
\item If $I_k \supset I_{k+1}$ then $p_{k+1}$ is one of the double cosets contained in $p_k$.
\end{itemize}
We write $p_{\bullet} \subset I_{\bullet}$. The final $(I_0, I_d)$-coset $p_d$ is called the \emph{terminus} of the path, and denoted $\term(p_{\bullet})$. 
A path $p_\bullet\subset I_\bullet$ is said to be \textit{forward} if every time that $I_k \supset I_{k+1}$ then $\ma{p_{k+1}}=\ma{p_k}$. Each expression has a unique forward path. One can prove that if $p_{\bullet}$ is the forward path of $I_{\bullet}$, then $I_{\bullet} \expr \term(p_{\bullet}).$
\end{defn}

We can also summarize the last two conditions 
above by saying that $p_k \cap p_{k+1}$ is nonempty. Paths subordinate to an expression $I_{\bullet}$ are the singular analogue of subexpressions, and the sequence $p_{\bullet}$ is analogous to the Bruhat stroll of \cite[Section 2.4]{Soergelcalculus}.

If $\ux'$ is a subexpression of $\ux$ and $\uy'$ is a subexpression of $\uy$, it is clear that the concatenation $\underline{x'y'}$ is a subexpression of the concatenation $\underline{xy}$. It is less obvious how to ``concatenate'' two subordinate paths.
Recall from Definition~\ref{Coxmon} the star product. In \cite{EKo} a $*$-product was defined on double cosets. If $p$ is an $(I,J)$-coset and $q$ is a $(J,K)$-coset then $p * q$ is the $(I,K)$-coset satisfying $\ma{(p * q)} = \ma{p} * \ma{q}$ (see \cite[Lemma 2.7]{EKo} to see that this is well-defined). 

\begin{lem} \label{lem:inclusions} Let $p$ be an $(I,J)$-coset, $q$ be a $(J,K)$-coset, and $q'$ be a $(J,K')$-coset, with $K' \subset K$ and $q' \subset q$. Then $p * q' \subset p * q$. \end{lem}

\begin{proof} By \cite[Lemma 2.15]{ EKo} we know that $\ma{q} = \ma{q'} . y$ for some $y \in W_{K}$. Then,  by \cite[Lemmas 2.3 and 2.2]{EKo}, $\ma{p} * \ma{q} = \ma{p} * \ma{q'} * y = (\ma{p} * \ma{q'}) . y'$ for some $y' \le y$. In particular, $\ma{p} * \ma{q}$ and $\ma{p} * \ma{q'}$ are in the same right $W_{K}$-coset, implying $p * q' \subset p * q$. \end{proof}

\begin{defn}\label{def.concatpath} Let $P_{\bullet} = [P_0, P_1, \ldots, P_c]$ be an expression and $p_{\bullet}$ a subordinate path with terminus $p$, and similarly for $Q_{\bullet} = [Q_0, \ldots, Q_d]$, $q_{\bullet}$ and $q$. Suppose $P_c = Q_0$, so the composition $P_{\bullet} \circ Q_{\bullet}$ exists. Then define the concatenation $p_{\bullet} \circ q_{\bullet}$ as the sequence of cosets
\begin{equation} [p_0, p_1, \ldots, p_c = p = p * q_0, p * q_1, \ldots, p * q_d = p * q]. \end{equation}
\end{defn}

Note that $p * q_0 = p$ since $q_0 = q_{\id}$ is the identity $(P_c,P_c)$-coset, which acts as the identity for the $*$-product  (recall that for any $J\subset S,$ we have $w_J*w_J=w_J$). We remark that if $P_{\bullet}\expr p$ and $Q_{\bullet}\expr q$ then by definition of the star product on cosets and by Equation~\eqref{expressp}, we have $P_{\bullet} \circ Q_{\bullet}\expr p*q$.

\begin{lem}\label{lem.pqsubPQ} The sequence $p_{\bullet} \circ q_{\bullet}$ is a path subordinate to $P_{\bullet} \circ Q_{\bullet}$. \end{lem}

\begin{proof} We need to verify that each term of $p_{\bullet} \circ q_{\bullet}$ is a coset of the appropriate kind, and that the intersection of two adjacent terms in the sequence is nonempty. The only interesting part is to prove that $(p * q_i) \cap (p * q_{i+1})$ is nonempty. Either $q_i \subset q_{i+1}$ or $q_{i+1} \subset q_i$, and the result follows from Lemma~\ref{lem:inclusions} either way. \end{proof}

\begin{lem}  A concatenation of forward paths is a forward path. \end{lem}

\begin{proof} This is obvious from the definitions. \end{proof}

\begin{rem} The converse to Lemma \ref{lem.pqsubPQ} is false: not every path subordinate to $P_{\bullet} \circ Q_{\bullet}$ has the form $p_{\bullet} \circ q_{\bullet}$. For example, there is a unique path subordinate to $[\mt,s]$ and a unique path subordinate to $[s,\mt]$, but there are two paths subordinate to $[\mt,s,\mt]$. Of these two, only the forward path is a concatenation, as forward paths are always guaranteed to be. \end{rem}

\subsection{Termini and concatenation}

Let us consider the set of termini
\begin{equation}
    \Term(I_\bullet) = \{\term(q_\bullet)\in W_I\backslash W / W_J\ |\ q_\bullet\subset I_\bullet\}
\end{equation}
of an expression $I_\bullet\expr p\in W_I\backslash W / W_J$.

Definition~\ref{def.concatpath} and Lemma~\ref{lem.pqsubPQ} say that
\begin{equation}\label{Term*Term}
\Term(I_\bullet)*\Term(J_\bullet)\subset \Term(I_\bullet\circ J_\bullet)
\end{equation}
holds for composable expressions $I_\bullet, J_\bullet$.
The inequality in \eqref{Term*Term} may be strict. For example, we have \[\Term([I,Is,I])=W_I\backslash W_{Is}/ W_I\]
while
\[\Term([I,Is])*\Term([Is,I])=\{W_IW_{Is}\}*\{W_{Is}W_I\}=\{W_Iw_{Is}W_I\}.\] 
More generally, we have the following comparison result.

\begin{lem}
Let $[I_0,\ldots,I_m]$ be a $(I,J)$-expression. Then
\begin{equation}\label{Termup}
    \Term(I_\bullet\circ[J,Js]) = \Term(I_\bullet)*(W_JW_{Js})
\end{equation} 
and 
\begin{align}\label{Termdown}
    \Term(I_\bullet\circ[J,J\setminus t]) &= \{q\in W_I\backslash W/W_{J\setminus t}\ |\  q\subset p\in\Term(I_\bullet)\}\\
    &\supset \{W_I\ma{p} W_{J\setminus t}\ |\ p\in \Term(I_\bullet)\}
    =
    \Term(I_\bullet)*(W_JW_{J\setminus t}).\nonumber
\end{align} 
\end{lem}
\begin{proof}
Let us first prove \eqref{Termup}. By  \eqref{Term*Term} it is enough to show `$\subset$'. Let $m$ be the width of $I_\bullet$, i.e., $I_\bullet = [I_0,\cdots,I_m]$. If $[p_0,\cdots,p_{m+1}]\subset I_\bullet\circ [J,Js]$ then $[p_0,\cdots p_m]\subset I_\bullet$ and thus $p_m\in\Term(I_\bullet)$.
Since $p_{m+1}=p_m*(W_JW_{Js})$ we have $\term(p_\bullet)\in \Term(I_\bullet)*(W_JW_{Js})$ as desired.

In \eqref{Termdown}, 
the first equality follows from the third bullet in Definition~\ref{defn:pathsubord},  and the last equality follows from the definition of the $*$-product of double cosets.
\end{proof}

\subsection{Bruhat order and subordinate paths}

\begin{thm}\label{newbruhatthm}
    Let $I,J\subset S$ be finitary and let $p\expr I_\bullet$ be a $(I,J)$-expression. Then we have
    \begin{equation}\label{eq.lep}
        \Term(I_\bullet) 
        =\lep.
    \end{equation}
\end{thm}
\begin{proof}
    We prove \eqref{eq.lep} by induction on the width of $I_\bullet$.
    We assume $\Term(I_\bullet)=\lep$, where $p\expr I_\bullet=[I_0,\cdots,I_m]$ is a $(I,J)$-expression, and show that both  
    \begin{equation}\label{case.p+s}
     \Term(I_\bullet\circ [J,Js])= \{q\ |\ q\leq pW_{Js}=W_IpW_{Js}\}
    \end{equation}
    and 
    \begin{equation}\label{case.p-s}
        \Term(I_\bullet\circ [J,J\setminus t])= \{q\ |\ q\leq W_I\ma{p}W_{J\setminus t}\}
    \end{equation} 
    hold.
    The base case $p=W_I W_I\expr [I]$ does satisfy $\Term([I])=\{p\}=\lep$.

     For \eqref{case.p+s}, consider the quotient map
     $\pi:W_I\backslash W/W_J\to W_I\backslash W/W_{Js}$, with which we can write
     $\Term(I_\bullet)*(W_JW_{Js})=\pi(\Term(I_\bullet))$. 
     Then we have 
     \[\Term(I_\bullet\circ[J,Js])=\Term(I_\bullet)*(W_JW_{Js})=\pi(\Term(I_\bullet))=\pi(\lep)=\ \leq \pi(p)\]
    where we use \eqref{Termup} in the first equality and Lemma~\ref{lem.posetprojection} in the last equality. This proves \eqref{case.p+s} since $\pi(p)=pW_{Js}$.

    For \eqref{case.p-s}, consider the quotient map 
     $\pi:W_I\backslash W/W_{J\setminus t}\to W_I\backslash W/W_J$. Then we have
     \[\Term(I_\bullet\circ[J,J\setminus t]) = 
     \pi\inv(\Term(I_\bullet))=\pi\inv(\lep)=\{ \leq \pi\inv(p)\}\]
    where we use again Lemma~\ref{lem.posetprojection} in the last equality. We complete the proof by observing that $\pi\inv(p)$ has a unique maximal element $W_I\ma{p}W_{J\setminus t}$.
\end{proof}

\begin{thm}\label{thm:bruhat}
Let $p,q$ be $(I,J)$-cosets, for fixed finitary subsets $I,J\subset S$. The following conditions are equivalent.
\begin{enumerate}
    \item\label{bruhatma} $\ma{p}\leq \ma{q}$ in $W$.
    \item\label{bruhatmi} $\mi{p}\leq \mi{q}$ in $W$ (i.e., $p\leq q$).
    \item\label{bruhat1} There exists  a reduced expression $I_\bullet \expr q$ and a subordinate path $p_\bullet\subset I_\bullet$ such that $p=\term(p_\bullet)$.
    \item\label{bruhatnonrex2} For any expression $I_\bullet \expr q$, there exists a subordinate path $p_\bullet\subset I_\bullet$ such that $p=\term(p_\bullet)$.
\end{enumerate}
\end{thm}
\begin{proof}
    The equivalence between \eqref{bruhatma} and \eqref{bruhatmi} is a special case of \Cref{lem.posetprojection} with the finitary subsets $\emptyset\subset I,\emptyset\subset J$. In fact, if $p\leq q$ then $\ma{p}\leq x$ for some $x\in q$ so $\ma{p}\leq x\leq \ma{q}$ holds.

    The conditions \eqref{bruhat1} and \eqref{bruhatnonrex2} are equivalent to \eqref{bruhatmi} by \Cref{newbruhatthm}, since the coset $q$ always has a reduced expression.
\end{proof}    

Thus the Bruhat order on double cosets has an equivalent definition: $p \le q$ if $p$ is the terminus of a path subordinate to some (or any) reduced (or not reduced) expression of $q$.

\subsection{Bruhat order and concatenation}

\begin{prop}\label{prop:bruhatconcatcompat} Suppose that $q$ and $q'$ are $(I,J)$-cosets with $q' \leq q$, and let $p$ be a $(K,I)$-coset and $r$ a $(J,L)$-coset. We have
\begin{equation} p * q' * r \leq p * q * r. \end{equation}
\end{prop}

\begin{proof} Pick expressions $P_{\bullet}$ and $Q_{\bullet}$ and $R_{\bullet}$ for $p$, $q$, and $r$ respectively. Then $P_{\bullet} \circ Q_{\bullet} \circ R_{\bullet}$ is an expression for $p*q*r$.  Let 
$q'_{\bullet} $ be a path subordinate to $Q_{\bullet}$ with terminus $q'$. Let $p_{\bullet}$ and $r_{\bullet}$ be the forward paths of $P_{\bullet}$ and $R_{\bullet}$.
The concatenation $p_{\bullet} \circ q'_{\bullet} \circ r_{\bullet}$ is a path subordinate to $P_{\bullet} \circ Q_{\bullet} \circ R_{\bullet}$ whose terminus is $p * q' * r$. Now the result follows from Theorem~\ref{thm:bruhat}. \end{proof}

We cannot replace `$\leq$' by  `$<$' in Proposition~\ref{prop:bruhatconcatcompat}, since the $*$-product 
with a double coset is not invertible:

\begin{ex}\label{ex.not<}
Let $r=W_I W$ be the unique $(I,S)$-coset and $p,q',q$ be as in Proposition~\ref{prop:bruhatconcatcompat}. Then we have $p*q'*r=p*q*r$ even if $q'<q$. 
\end{ex}

A strict version of Proposition~\ref{prop:bruhatconcatcompat} holds when the composition $p*q*r$ is reduced in the following sense.
Recall first that for $x, y \in W$ we write $xy = x.y$ if $\ell(x) + \ell(y) = xy$,  and we call the composition \emph{reduced}. 

\begin{notation} \label{dotforcosets} 
For an 
$(I,J)$-coset $p$, a $(J,K)$-coset $q$, and a $(I,K)$-coset $r$, let us write $p . q = r$ if there exist $I_{\bullet}$ and $K_{\bullet}$ such that
\[ I_{\bullet} \expr p, \quad K_{\bullet} \expr q, \quad I_{\bullet} \circ K_{\bullet} \expr r\]
are all reduced expressions. We say that $r$ is a \emph{reduced composition} of $p$ and $q$.  By \cite[Proposition 4.3]{EKo}, we have $p . q = r$ if and only if $\ma{r} = \ma{p}. (w_J^{-1} \ma{q}) = (\ma{p} w_J^{-1}) . \ma{q}$.
\end{notation}

\begin{thm}  \label{thm:bruhatconcatcompat} Suppose that $q$ and $q'$ are $(I,J)$-cosets with $q' < q$, and let $p$ be a $(K,I)$-coset and $r$ a $(J,L)$-coset. If $p . q . r$ is a reduced composition, then 
\begin{equation} p * q' * r < p . q . r. \end{equation}
\end{thm}

\begin{proof} We have $p * q' * r \leq p . q . r$ by Proposition~\ref{prop:bruhatconcatcompat}. 
Thus the claim follows from
\begin{align*}
    \ell( \ma{p * q' * r}) &\leq \ell(\ma{p})-\ell(J)+\ell(\ma{q'})-\ell(I)+\ell(\ma{r})\\&  <\ell(\ma{p})-\ell(J)+\ell(\ma{q})-\ell(I)+\ell(\ma{r}) =\ell(\ma{p . q . r}).\qedhere
\end{align*}
\end{proof}

\subsection{Bruhat order, length, and reduced expressions}
We wish to record some easy consequences of the results above about Bruhat order, for ease of future use.

\begin{lem} The Bruhat order on double cosets respects length: if $q \le p$ then $\ell(q) \le \ell(p)$, with equality if and only if $q = p$. \end{lem}

\begin{proof} Let $p$ be an $(I,J)$-coset and recall that $\ell(p) = 2 \ell(\ma{p}) - \ell(I) - \ell(J)$. If $q$ is an $(I,J)$-coset with $q \le p$ then $\ma{q} \le \ma{p}$, with equality if and only if $q = p$. Now the result follows from the corresponding fact for the ordinary Bruhat order: $\ell(\ma{q}) \le \ell(\ma{p})$ with equality if and only if $\ma{q} = \ma{p}$. \end{proof}

Below we use 
that contiguous subexpressions of reduced expressions are reduced, see \cite[Proposition 3.12]{EKo}.

\begin{lem} Let $I_{\bullet} = [I_0, \ldots, I_d]$ be a reduced expression for $p$. Then it has a unique subordinate path with terminus $p$, namely the forward path. \end{lem}

\begin{proof}
    We prove the claim by induction on the width $d$ of a reduced expression $p\expr [I_0,\ldots,I_d]$, where the base case is trivial.
    Let $p_\bullet$ be the forward path and let $t_\bullet$ be another subordinate path with terminus $p$. 
    If $t_{d-1} = p_{d-1}$ then, since $[I_0, \ldots, I_{d-1}]$ is reduced, induction implies that $t_{\le d-1} = p_{\le d-1}$. Combined with $t_d = p_d = p$, we have $t_{\bullet} = p_{\bullet}$.
    
    If $t_{d-1}\neq p_{d-1}$ then $t_{d-1}<p_{d-1}$ (see Theorem~\ref{newbruhatthm}). For $t_d=p_d$ to be the case the only possibility is $I_{d-1} \subset I_d$ and $t_{d-1}\subset t_d=p_d = p$. But since $p_\bullet$ is reduced we have $\mi{p_{d-1}}=\mi{p_d}$ which contradicts $\mi{t_{d-1}}<\mi{p_{d-1}}$. Thus $t_{d-1} = p_{d-1}$.
\end{proof}

\section{Ideals of lower terms in the singular Hecke 2-category} \label{sec:lower}

The singular Hecke 2-category is a categorification of the Hecke (or Schur) algebroid (see \cite[Definition 2.7]{SingSb}). Like most categorifications, it has several incarnations, all of which are isomorphic in non-degenerate characteristic zero situations, but which may differ in general. Currently, in the literature, there is a geometric incarnation (using perverse sheaves on partial flag varieties, equivariant under parabolic subgroups of a Lie group, see \cite[p8-10]{WThesis}), and an algebraic incarnation using singular Soergel bimodules (which are direct summands of singular Bott-Samelson bimodules, to be defined below). The latter is the topic of Williamson's PhD thesis \cite{WThesis}, which also appears in a shortened article version \cite{SingSb}.

One can also expect a diagrammatic incarnation by generators and relations, following the rubric set out in \cite{ESW}, see also \cite[Chapter 24]{GBM}. The diagrammatic version has not been fully developed. In the interest in developing it further, it is important to establish combinatorially some basic facts about morphisms in the 2-category. In this paper, we address the ideal of lower terms.

We fix a Coxeter system $(W,S)$ and a realization $V$ over a ground field $\Bbbk$ thereof \cite[Section 3.1]{Soergelcalculus}. We consider the polynomial ring $R$ of this realization, and for each finitary subset $I \subset S$, the subring $R^I$ of $W_I$-invariants in $R$. We require that our realization is reflection faithful, balanced and satisfies generalized Demazure surjectivity (cf. \cite[\S 3.1]{EKLP1}). The ring $R$ is graded, and all $R^I$-modules 
are graded.
The background on this material in \cite[Chapter 3.1]{EKLP1} should be sufficient.

 In \S\ref{SSBimBackground} we provide background on singular Bott--Samelson and singular Soergel bimodules. In \S\ref{ssec:lowertermsdefined} we define the ideal of lower terms using factorization, and in \S\ref{ss.locality} we prove our result on the compatibility of ideals and the monoidal structure. These results rely upon \Cref{ass}, which describes certain properties of singular Soergel bimodules that we prove in \S\ref{ss.character}. However, the properties of \Cref{ass} can be viewed as a black box for the purposes of our monoidal compatibility result. Eventually, these properties shall be proven separately for the diagrammatic category, at which point we can adapt some of our results to that context.

Starting in \S\ref{ss.character} we focus on singular Soergel bimodules in their algebraic incarnation. Keep in mind that there are two different kinds of filtrations at play: a filtration (by bimodules) of the bimodules themselves, and a filtration (by ideals) on morphisms between bimodules. Both of these filtrations will be given two definitions, one involving support, and one involving factorization. First we establish the equivalence of these definitions for bimodule filtrations, and then we translate the result to morphism spaces.

In \S\ref{ss.character} we recall facts about $\nabla$-filtrations and support filtrations on singular Soergel bimodules from \cite[\S 6.1]{SingSb}, which we use to prove \Cref{ass}. In \S\ref{ss.supportvsfactor} we state the equivalence to a factorization filtration. We prove this result in \S\ref{ss.restrictionfun}, by translating the problem to the setting of ordinary Soergel bimodules.

The crucial tool we use to transfer results about bimodule filtrations to results about morphism filtrations is the existence of resolutions of $\nabla$-filtered modules by certain Soergel bimodules. We prove this in \S\ref{ss.nowformorphisms} using a horseshoe-lemma-style argument. Once this is done, we have a filtration of the ordinary Hecke category by ideals, defined equivalently using support conditions or factorization conditions.

Finally, in \S\ref{ss.application2} we translate this result to the singular Hecke category. Then we use these techniques to analyze what happens to the one-tensor, the minimal degree generator of a Bott--Samelson bimodule.

\subsection{Singular Bott-Samelson bimodules} \label{SSBimBackground}

\begin{defn} If $I\subset J$ then we define the graded  $(R^I, R^J)$-bimodule 
$$\BS([I,J]) = R^I, $$
and the graded $(R^J, R^I)$-bimodule 
$$\BS([J,I]) = R^I(\ell(J) - \ell(I)).  $$ 
By convention, the grading shift by the positive integer $\ell(J) - \ell(I)$ places $1 \in R^I$ in negative degree $\ell(I) - \ell(J)$. \end{defn}

\begin{defn} We define the \emph{singular Bott--Samelson bimodule} $\BS(I_{\bullet})$ associated to a singular expression $I_{\bullet} = [I_0, \ldots, I_d]$ as the graded $(R^{I_0}, R^{I_d})$-bimodule
\begin{equation} \BS(I_{\bullet}) = \BS([I_0,I_1]) \ot_{R^{I_1}} \BS([I_1, I_2]) \ot_{R^{I_2}} \cdots \ot_{R^{I_{d-1}}} \BS([I_{d-1},I_d]). \end{equation}
\end{defn}

We denote by $\onetensor_{L_\bullet}$ (or simply by $\onetensor$) the element \[\onetensor_{L_\bullet} :=1\otimes 1 \otimes \ldots \otimes 1\in \BS(L_\bullet).\]


We work in the bicategory of graded bimodules, looking at $(R^I, R^J)$-bimodules for various finitary $I, J \subset S$. To be more precise, denote by $\Bim$ the bicategory defined as follows. The objects in $\Bim$ are the finitary subsets $I\subset S$, identified with the graded algebras $R^I$. The category $\Bim(J,I)$ of 1-morphisms between $I,J$ is the category of graded $(R^I,R^J)$-bimodules. The composition of $1$-morphisms $\Bim(J,I)\times \Bim(K,J)\to \Bim(K,I)$ is given by the tensor product over $R^J$.
Given $M,N\in\Bim(J,I)$, the morphism space is the graded $(R^I,R^J)$-bimodule, 
\begin{equation}\label{eq.Hom}
    \Hom(M,N):=\bigoplus_{i\in \mathbb Z}\Hom^i(M,N), \qquad \Hom^i(M,N):=\Hom^0(M,N(i)).   
    \end{equation}
    Here $\Hom^0(M,N(i))$ denotes the space of degree zero $(R^I,R^J)$-bimodule maps from $M$ to $N(i)$. 
    
Bott--Samelson bimodules form a subbicategory $\SBSBim$, where $\SBSBim(J,I)$ is the category of Bott--Samelson bimodules associated to $(I,J)$-expressions.

The assumptions about Bott-Samelson bimodules we need fit into the following black box.

\begin{prop}\label{ass} 
For each $(I,J)$-coset $p$ there exists an indecomposable graded $(R^I, R^J)$-bimodule $B_p$, which is a direct summand with multiplicity one of $\BS(I_{\bullet})$ for any reduced expression $I_{\bullet} \expr p$, but not (isomorphic to) a direct summand of $\BS(I'_{\bullet})$ whenever $\ell(I'_{\bullet}) < \ell(p)$.
Moreover, the summand $B_p$ contains the element $\onetensor_{I_\bullet}$.
Every indecomposable summand of a singular Bott-Samelson bimodule is isomorphic to some $B_p$, up to grading shift. Finally, if $I_{\bullet}$ is an expression (not necessarily reduced) for a coset $p$, then every summand of $\BS(I_{\bullet})$ is isomorphic to a grading shift of $B_q$ for some $q \le p$. 
\end{prop}

\begin{proof}
A proof of the first two claims can be found in \cite[Theorem 7.10]{SingSb}. A proof of the third one can be found in the proof of the same theorem. 
The last claim is proved using results in \cite{SingSb} on the character of Soergel bimodules. We give the latter proof in Section~\ref{ss.character} after discussing necessary definitions.
\end{proof}

The bicategory of singular Soergel bimodules $\SSBim$ is the additive closure of $\SBSBim$ in $\Bim$. Proposition~\ref{ass} says that the indecomposable 1-morphisms in $\SSBim$, up to isomorphism and grading shift, are the bimodules $B_p$ indexed by double cosets $p$.

\begin{cor} \label{cor:factorthroughnonreduced} Let $I_{\bullet}$ be an expression, not necessarily reduced, for a coset $p$. Then every morphism which factors through $\BS(I_{\bullet})$ is a finite sum of morphisms factoring through $\BS(M_{\bullet})$, for various reduced expressions $M_{\bullet} \expr p'$, for various cosets $p' \le p$. \end{cor}

\begin{proof} A morphism factoring through an object can be written as a finite sum of morphisms factoring through its direct summands, and vice versa. Thus a morphism which factors through $\BS(I_{\bullet})$ can be factored instead through $B_{p'}$ for various $p' \le p$ (up to taking finite sums). Morphisms factoring through $B_{p'}$ can be factored instead through $\BS(M_{\bullet})$ for a reduced expression. \end{proof}

\subsection{Lower terms} \label{ssec:lowertermsdefined}

Inside any additive category, given a collection of objects, the set of finite sums of morphisms, each of which factors through one of those objects, forms a two-sided ideal. This is the same as the ideal generated by the identity maps of those objects.

\begin{defn}
Let $C\subset W_I\backslash W/W_J$ be a downward-closed set (i.e., $p\in C$ and $q\leq p$ implies $q\in C$).
Consider the set of expressions $M_\bullet\expr q$ for $q \in C$. 
Let $\Hom_C$ denote the ideal generated by the identity maps of $\BS(M_{\bullet})$ for such expressions. 
This is a two-sided ideal in the category of $(R^I, R^J)$-bimodules.
\end{defn}

The left and right actions of $R^I$ and $R^J$ on $\Hom(B,B')$ are given by 
$(afb)(m)=f(amb)=af(m)b$, where $f\in\Hom(B,B')$, $a\in R^I$, and $b\in R^J$. 
Thus these actions
preserve the factorization of morphisms, i.e., $a(h\circ g)b=ahb\circ g = h\circ agb$ where $h,g$ are composable $(R^I,R^J)$-bimodule morphisms. Consequently, $\Hom_C(B,B') \subset \Hom(B,B')$ is a sub-bimodule. 

 Important special cases include $C = \{\le p\}$ and $C = \{< p\}$.

\begin{notation} Let $p$ be an $(I,J)$-coset. Then $\Hom_{< p}$ is called the \emph{ideal of lower terms} relative to $p$. \end{notation}



\subsection{The locality of lower terms}\label{ss.locality}

Here we prove that the concept of ``lower terms'' is preserved by tensor product with identity maps under reduced compositions of reduced expressions.


\begin{prop} \label{prop:lowertermslocal} Let $P_{\bullet} \expr p$ and $Q_{\bullet} \expr q$ and $R_{\bullet} \expr r$ be reduced expressions such that $P_{\bullet} \circ Q_{\bullet} \circ R_{\bullet} \expr p.q.r$ is reduced. Then
\begin{equation} \id_{\BS(P_{\bullet})} \ot \End_{< q}(\BS(Q_{\bullet})) \ot \id_{\BS(R_{\bullet})} \subset \End_{< p.q.r}(\BS(P_{\bullet} \circ Q_{\bullet} \circ R_{\bullet})). \end{equation} \end{prop}

\begin{proof} Let $\phi \in \End_{<q}(\BS(Q_{\bullet}))$. There exists some reduced expression $Q'_{\bullet} \expr q'$ with $q' < q$, such that $\phi$ factors through
$\BS(Q'_{\bullet})$. Then $\id_{\BS(P_{\bullet})} \ot \phi \ot \id_{\BS(R_{\bullet})}$ factors through the Bott--Samelson bimodule for $P_{\bullet} \circ Q'_{\bullet} \circ R_{\bullet}$. The latter need not be a reduced expression, but it is an expression for the double coset $p*q'*r$. By Theorem~\ref{thm:bruhatconcatcompat}, $p*q'*r < p.q.r$. By Corollary~\ref{cor:factorthroughnonreduced}, $\id_{\BS(P_{\bullet})} \ot \phi \ot \id_{\BS(R_{\bullet})}$ is a finite sum of morphisms factoring through  Bott-Samelsons expressing
cosets strictly smaller than $p.q.r$, whence it lives in $\End_{<p.q.r}(\BS(P_{\bullet} \circ Q_{\bullet} \circ R_{\bullet}))$. \end{proof}

\subsection{Standard bimodules and \texorpdfstring{$\nabla$-}{nabla }filtrations}\label{ss.character}

\begin{defn} [{\cite[Definition 4.4]{SingSb}}]
Let $p$ be a $(I,J)$-coset. The \emph{(singular) standard bimodule} $R_p$ is the graded $(R^I, R^J)$-bimodule defined as follows. 
We first define the \emph{left redundancy} $\leftred(p):= I \cap \mi{p}J\mi{p}\inv$ and the \emph{right redundancy} $\rightred(p):= J\cap \mi{p}\inv I \mi{p}$ of $p$. 
As a left $R^I$-module, $R_p$ is $R^{\leftred(p)}$. The right action of $f \in R^J$ on $m \in R_p$ is given by
\begin{equation}\label{mf=pfm} m \cdot f = (\mi{p} f) m.
\end{equation}
In words, we say that the right action is twisted by $\mi{p}$.

\end{defn}

For \eqref{mf=pfm} to make sense, we need $\mi{p} f \in R^{\leftred(p)}$ whenever $f \in R^J$. Note that $f \in R^J$ implies that $\mi{p} f$ is invariant under ${\mi{p} J \mi{p}^{-1}}$. By definition, $\leftred(p) \subset \mi{p} J \mi{p}^{-1}$. That \eqref{mf=pfm} is associative follows from $\mi{p}(f_1f_2)= (\mi{p}f_1)(\mi{p}f_2)$.

Equivalently, we can define $R_p$ to be $R^{\rightred(p)}$ as a right $R^J$-module. Then the left action of $g \in R^I$ on $n \in R^{\rightred(p)}$ is given by
\begin{equation} g \cdot n = n (\mi{p}^{-1} g). \end{equation}
These two descriptions are intertwined by the inverse isomorphisms between $R^{\leftred(p)}$ and $R^{\rightred(p)}$ given by 
\begin{equation} m \mapsto \mi{p}^{-1} m, \qquad n \mapsto \mi{p} n. \end{equation}

\begin{ex} For $s \ne t \in S$, let $p$ be the $(s,t)$-coset containing the identity. Then $R_p = R$ as a left $R^s$-module and a right $R^t$-module. \end{ex}

\begin{ex} Let $s \ne t \in S$ with $m_{st} = 3$. Let $p$ be the $(s,t)$-coset containing $sts$. Then $R_p = R^s$ where the right action of $R^t$ is twisted by $ts$. \end{ex}

\begin{ex} Let $s \ne t \in S$ with $m_{st} = 3$. Let $p$ be the $(s,s)$-coset containing $sts$. Then $R_p = R$ as an $R^s$-bimodule, where the right action of $R^s$ is twisted by $t$. \end{ex}

 In all the examples above, $R_p$ is cyclic as an $(R^I, R^J)$-bimodule, generated by the identity element $1$. This is not true in general.

\begin{ex}  This example paraphrases \cite[Example 2.3.6(3)]{WThesis}. Let $W = S_4$ and $R = \Bbbk[x_1, x_2, x_3, x_4]$. Let $I = \{s_1, s_3\} = J$, and $p$ be the $(I,J)$-coset with $\mi{p} = s_2$. Then $\leftred(p) = \mt$ so $R_p = R$ as a vector space. If $R_p$ were cyclic it would be generated by the identity element of $R$. The linear terms in $R^I \cdot 1$ are spanned by $\{x_1+x_2, x_3+x_4\}$, and the linear terms in $1 \cdot R^J$ are spanned by $\{x_1+x_3, x_2+x_4\}$. Together these only span a proper subspace of the linear terms in $R$. \end{ex}

Despite the lack of cyclicity, Williamson is able to prove the rather subtle result below.

\begin{lem}\label{lem:endstd} Viewing $R_p$ as a ring via the identification with $R^{\leftred(p)}$, the natural map $R_p \to \End_{(R^I, R^J)}(R_p)$ is an isomorphism of rings and of $(R^I, R^J)$-bimodules. Therefore, $R_p$ is indecomposable as a graded bimodule. \end{lem}

\begin{proof} The first statement follows from \cite[Corollary 4.13]{SingSb}. 
Since $R^{\leftred(p)}$ and hence $\End(R_p)$ is a graded local ring, $R_p$ is indecomposable. \end{proof}


\begin{lem}\label{lem:projstd}
    Let $M$ be a finite direct sum of copies of shifts of $R_p$ and let $N\subset M$ be a direct summand. Then $N$ is also a direct sum of copies of shifts of $R_p$.
\end{lem}

\begin{proof}
This follows from the fact that the endomorphism ring of $R_p$ is (graded) local, by the Krull-Schmidt theorem. The Krull-Schmidt theorem  is usually stated in the non-graded context, and says that if a module $M$ over a ring has a finite decomposition $M = \bigoplus M_i$ where each $M_i$ has a local endomorphism ring, then the $M_i$ are indecomposable, and $M$ has a unique decomposition into indecomposable summands. Thus any summand of $M$ must be a direct sum of some of the $M_i$. For the proof in this context, see \cite[Theorem 11.50]{GBM}. This proof adapts immediately to the context of graded modules over a graded ring, graded local rings, and homogeneous direct sum decompositions. The crucial technical point is that a sum of homogeneous non-invertible elements in a graded local ring is non-invertible. \end{proof}

\begin{defn}[{\cite[Definition 6.1]{SingSb}}] \label{defn:nablafilt}
Let $M$ be a graded $(R^I,R^J)$-bimodule finitely generated as a left $R^I$-module and a right $R^J$-module.
Fix an enumeration $p_1,p_2,\ldots$ of the elements $W_I\backslash W/W_J$ compatible with the Bruhat order, i.e. such that $p_i<p_j$ only if $i<j$. A \emph{$\nabla$-filtration} $F_\bullet$ on $M$ (associated to the given enumeration) is a filtration of graded $(R^I,R^J)$-bimodules such that $F_{i+1}(M)/F_i(M)$ is isomorphic to a direct sum of shift of standard bimodules $R_{p_i}$.
\end{defn}

To reiterate, a $\nabla$-filtration is not just a filtration where the subquotients are standard bimodules. One also requires that the cosets appearing as subquotients appear in an order compatible with the Bruhat order on double cosets. The subquotients of a $\nabla$-filtration are unique by \cite[Lemma 6.2]{SingSb}.

Following the recent erratum to \cite{SingSb}, for an $(I,J)$-coset $p$ we set
\begin{equation} \nabla_p := R_p(\ell(\ma{p}) - \ell(J)). \end{equation}
We can define the \emph{character} of a bimodule $B$ with a $\nabla$-filtration $F_\bullet$ as 
\begin{equation}\label{defchar}
     \ch{B}=\sum_{p\in W_I\backslash W/W_J}\overline{g_p(B)}{}^I\hspace{-0.05cm}H_p^J 
\end{equation}
where $g_p(B)\in \mathbb{Z}[v^{\pm 1}]$ is some polynomial encoding the multiplicity of $\nabla_p$ in $F_{i+1}(B)/F_i(B)$. Here ${}^I\hspace{-0.05cm}H_p^J$ is a an element of the standard basis for the morphism space from $J$ to $I$ in the Hecke algebroid. We will not need the details in this paper; it suffices to note several facts. First, $\ch{\nabla_p} = {}^I\hspace{-0.05cm}H_p^J.$ Second, whenever $I \subset J$, we have $\BS([[I \subset J]]) = \nabla_p$ where $p$ is the $(I,J)$-coset containing the identity, so $\ch{\BS([[I \subset J]])} = {}^I\hspace{-0.05cm}H_p^J$, which for this particular $p$ we abbreviate to ${}^I\hspace{-0.05cm}H^J$. Similarly, $\BS([[J \supset I]]) = \nabla_p$, where $p$ is the $(J,I)$-coset containing the identity, so $\ch{\BS([[J \supset I]])} = {}^J\hspace{-0.05cm}H_p^I$, which for this particular $p$ we abbreviate to ${}^J\hspace{-0.05cm}H^I$.
 

Bimodules with $\nabla$-filtrations are closed under tensoring with singular Bott--Samelson bimodules (\cite[Theorem 6.4]{SingSb}). In particular, Bott--Samelson bimodules have $\nabla$-filtrations. However, if $B$ admits a $\nabla$-filtration, it is not obvious that a direct summand of $B$ should admit a $\nabla$-filtration. An explicit construction of $\nabla$-filtrations on Bott--Samelson bimodules and their summands is found in \cite[\S 4.5]{SingSb}, and we recall it now in a more general context of functorial $\nabla$-filtrations. This discussion of functoriality is not found quite so explicitly in Williamson's work.

\begin{defn} A \emph{filtration of the identity functor} of $\Bim(J,I)$ is a family $\{F_C\}$ of additive endofunctors of $\Bim(J,I)$, indexed by downward-closed subsets $C\subset W_I\backslash W/W_J$ in the Bruhat order. Each functor $F_C$ is a subfunctor of the identity functor, in that $F_C(M) \subset M$, and for a morphism $f \co M \to N$, $F_C(f)$ is just the restriction of $f$ to $F_C(M)$, which has image in $F_C(N)$. Moreover, whenever $C \subset C'$ we have inclusions $F_C(M) \subset F_{C'}(M)$. \end{defn}

Given $C\subset C'$ the quotient $F_{C'}/F_{C} (M):=F_{C'}(M)/F_C(M)$ defines an additive endofunctor on $\Bim(J,I)$.

\begin{defn} Let $\{F_C\}$ be a filtration of the identity functor. Given an enumeration $\{p_i\}$ of the double cosets, we say that $\{F_C\}$ \emph{induces a functorial $\nabla$-filtration} on a bimodule $M$ if 
\begin{equation}\label{inotation}
F\lei(M):= F_{\{p_j \mid j \le i\}}(M)     
\end{equation}
is a $\nabla$-filtration on $M$. \end{defn}

The functoriality of the filtration $F\lei(M)$ ensures a compatibility of this $\nabla$-filtration with direct summands.

\begin{lem} \label{summandtosummand} Let $\{F_C\}$ be a filtration of the identity functor. Let $M$ be a bimodule, and $N$ a direct summand of $M$, the image of an idempotent $e \in \End(M)$. Then $F_C(N)$ is a direct summand of $F_C(M)$, the image of $e$ restricted to $F_C(M)$. \end{lem}

\begin{proof} The same is true of any additive subfunctor of the identity, but let us spell it out for the reader. The bimodule $N$ is a summand of $M$ if and only if there is a projection map $p \co M \to N$ and an inclusion map $\iota \co N \to M$ such that $p \circ \iota = \id_N$. Note that $N$ is the image of $e = \iota \circ p$.
Now we have 
\[F_C(p)\circ F_C(\iota)=F_C(p\circ\iota)=F_C(\id_N) = \id_{F_C(N)},\]
making $F_C(N)$ a direct summand of $F_C(M)$. 
\end{proof}

\begin{lem} \label{lem:functorialtosummand} Let $\{F_C\}$ be a filtration of the identity functor, which induces a $\nabla$-filtration on a bimodule $M$. Let $N$ be isomorphic to
a direct summand of $M$. Then $\{F_C\}$ induces a $\nabla$-filtration on $N$. \end{lem}

\begin{proof} As in the proof of Lemma~\ref{summandtosummand}, $F_C(N)$ is (isomorphic to) a direct summand of $F_C(M)$, via maps $p:M\to N$ and $\iota:N\to M$.
Then $(F_i/F_{i-1})(p)$ and $(F_i/F_{i-1})(\iota)$ compose to the identity map on $F_i(N)/F_{i-1}(N)$ making the latter a direct summand of $F_i(M)/F_{i-1}(M)$. 
By Lemma \ref{lem:projstd}, $F_i(N)/F_{i-1}(N)$ is also a direct sum of shifts of $R_{p_i}$. Thus $\{F_i(N)\}$ is a $\nabla$-filtration. 
\end{proof}

\begin{lem} \label{lem:soequalonsummands} Let $\{F_C\}$ and $\{G_C\}$ be filtrations of the identity functor. Fix a bimodule $M$, and suppose $F_C(M) = G_C(M)$ as submodules of $M$, for all $C$. If $N$ is a direct summand of $M$, then $F_C(N) = G_C(N)$ for all $C$. \end{lem}

\begin{proof} If $N$ is the image of $e$ inside $M$, then $F_C(N)$ is the image of $e$ restricted to $F_C(M)$, and $G_C(N)$ is the image of $e$ restricted to $G_C(M)$. Since $F_C(M) = G_C(M)$, we have $F_C(N) = G_C(N)$. \end{proof}

Now we recall Williamson's functorial $\nabla$-filtration. Let $B\in \Bim(J,I)$. Then we can regard $B$ as a $R^I\otimes_\Bbbk R^J$-module and thus as (the global sections of) a quasi-coherent sheaf on $V/W_I\times V/W_J$, where $V$ is our realization of $W$. 


\begin{defn}\label{def.Gamma}
    Let $C\subset W_I\backslash W/W_J$ be a downward-closed set and $B\in \Bim(J,I)$.
    Regarding $B$ as a quasi-coherent sheaf on $V/W_I\times V/W_J$, 
    set $\Gamma_C(B)\subset B$ to be the submodule consisting of the elements supported on the graph of $C$, or to be more precise, supported on the image of $\{(xv,v)\ |\ x\in\pi\inv C\subset W, v\in V\}$ under the projection $V\times V\to V/W_I\times V/W_J$.
    This defines a left exact functor $\Gamma_C:\Bim(J,I)\to \Bim(J,I)$  and a filtration $\{\Gamma_C\}$ of the identity functor, called the \emph{support filtration}.
\end{defn}

\begin{prop}\label{supportfiltr} 
If $B\in \SSBim(J,I)$, then the support filtration $\{\Gamma_C\}$
is a functorial $\nabla$-filtration on $B$ (for any enumeration of $W_I\backslash W/W_J$ compatible with the Bruhat order).
\end{prop}


\begin{proof}
For a Bott--Samelson bimodule $B\in\BSBim(J,I)$, this follows from \cite[Lemma 6.2 and Theorem 6.4]{SingSb}. Any singular Soergel bimodule is a direct summand of a Bott--Samelson bimodule, so by \Cref{lem:functorialtosummand} we obtain the result for any singular Soergel bimodule. \end{proof}



Proposition~\ref{supportfiltr} implies in particular that the character \eqref{defchar} is well-defined for singular Soergel bimodules.
Now we can finish the proof of \Cref{ass}.

\begin{proof}[Proof of the last claim in Proposition~\ref{ass}]
For $p$ an $(I,J)$-coset, let the reader recall from \cite[p. 4569]{SingSb} the definitions of the standard basis ${}^I\hspace{-0.05cm}H_p^J$ and of the standard generators ${}^I\hspace{-0.05cm}H^J$ in the Hecke algebroid over the ring $\mathbb{Z}[v,v^{-1}].
 $
 By \cite[Theorem 7.10]{SingSb} we have \begin{equation} \label{eq:chbq} \mathrm{ch}(B_q)\in {}^I\hspace{-0.05cm}H_q^J+\sum_{r<q}\mathbb{Z}[v,v^{-1}]\,\cdot\, {}^I\hspace{-0.05cm}{H}_r^J. \end{equation}
%
Since the characters of the indecomposable bimodules $B_q$ form a basis for the Hecke algebroid, the decomposition of a Bott-Samelson bimodule is determined by its character. Thus, to prove the third claim it is enough to show that 
\begin{equation}\label{character}
    \mathrm{ch}(\BS(I_{\bullet}))\in \sum_{q\leq p}\mathbb{Z}[v,v^{-1}]\,\cdot\, {}^I\hspace{-0.05cm}{H}_q^J
\end{equation}
whenever $I_{\bullet}$ is an expression (not necessarily reduced) for $p$.

If $I_{\bullet} = [I_0, \ldots, I_d]$ then $\ch{\BS(I_{\bullet})}$ is an iterated product of ${}^{I_j}\hspace{-0.05cm}{H}^{I_{j+1}}$. By an iterated application of \cite[Proposition 2.8]{SingSb}, the result is a linear combination of ${}^I\hspace{-0.05cm}{H}_q^J$ over those cosets $q$ which are termini of paths subordinate to $I_{\bullet}$ (to see this spelled out in more detail see \cite[\S 3.3]{EKLP2}. Then the result follows from Theorem \ref{thm:bruhat}.
\end{proof}

\subsection{Support filtrations and lower terms} \label{ss.supportvsfactor}

In this section we give an alternative construction of the support filtration.

\begin{defn}
Given a downward-closed set $C\subset W_I\backslash W/W_J$ and a $(R^I,R^J)$-bimodule $B$, we set
\begin{equation*}
    N_C(B):=\sum_{f\in \Hom(B_q,B), q\in C}
\im(f).
\end{equation*}
\end{defn}

If $f \in \Hom(B_q,B)$ and $g \in \Hom(B,B')$ then $g \circ f \in \Hom(B_q,B')$. From this it is easy to see that $N_C$ is a subfunctor of the identity functor, and $\{N_C\}$ is a filtration of the identity functor.

\begin{lem}\label{NsubGamma} For any $B \in \Bim(J,I)$ and any downward-closed subset $C$ of the Bruhat order on $(I,J)$-cosets, we have $N_C(B) \subset \Gamma_C(B)$. \end{lem}

\begin{proof} If $q \in C$ then $B_q = \Gamma_C(B_q)$. Thus any map $B_q \to B$ has image in $\Gamma_C(B)$. \end{proof}

The following is a generalization of \cite[Lemma 4.7]{PatRouquier} from the case of one-sided singular Soergel bimodules.

\begin{prop}\label{N=Gamma}
    For $B\in \SSBim(J,I)$ and a downward-closed finite subset $C\subset W_I\backslash W/W_J$ we have $N_C(B)=\Gamma_C(B)$.
\end{prop}

We prove Proposition~\ref{N=Gamma} in the next section. Our method of proof will be to reduce to the case $(I,J) = (\mt, \mt)$, where the proof follows from \cite[Proposition 12.26]{GBM} and \cite[Theorem 10.32]{GBM}. First we recall the well-known fact that any bimodule in $\SSBim(J,I)$ is a direct summand of a bimodule restricted from an ordinary Soergel bimodule in $\SSBim(\mt, \mt)$. Thus by Lemma \ref{lem:soequalonsummands} we need only prove the result for bimodules restricted from $\SSBim(\mt, \mt)$. In the next section we study the compatibility of restriction with the functors $\Gamma_C$ and $N_C$, and deduce the result.

\subsection{Restricting from regular Soergel bimodules} \label{ss.restrictionfun}

A regular expression \[w\expr[s_1,\cdots ,s_l]\] of an element $w\in W$ is viewed as the $(\emptyset,\emptyset)$-expression 
\[\{w\}\expr [\emptyset,\{s_1\},\emptyset,\cdots,\{s_l\},\emptyset].\] In this way, we have an embedding of the category of regular Bott-Samelson bimodules $\BSBim$ to $\SBSBim(\emptyset,\emptyset)$. Proposition~\ref{ass} says that this identifies the category of regular Soergel bimodules $\SBim$ with $\SSBim(\emptyset,\emptyset)$.
We make use of (restrictions of) the following functors to import results on $\SBim=\SSBim(\emptyset,\emptyset)$ to our $\SSBim(J,I)$  for general (finitary) $I,J\subset S$.
We fix $I,J\subset S$ and use the shorthand
$\res$ for the restriction functor
\[ \ _{R^I}-_{R^J}:\Bim(\emptyset,\emptyset)\to \Bim(J,I) \]
and $\ind$ for the induction
\[R\otimes_{R^I} - \otimes_{R^J}R:\Bim(J,I)\to \Bim(\emptyset,\emptyset)\]
as well as their restriction to Bott-Samelson or Soergel bimodules. 
In this section we denote by $\pi:W\to W_I\backslash W/ W_J$ the quotient map discussed in Lemma~\ref{lem.posetprojection}.
Lemma~\ref{lem.posetprojection} says in particular that $\pi\inv C$ is downward-closed if $C\subset W_I\backslash W/ W_J$ is downward-closed.

For finitary subsets $I,J\subset S$ and a downward-closed $C\subset W_I\backslash W/ W_J$, we denote by $\SSBim_C(J,I)$ the full additive subcategory of $\SSBim(J,I)$ generated by (shifts of) $B_p$ for $p\in C$.

\begin{lem}\label{BinOrdinarySB}
\begin{enumerate}
    \item The identity functor on $\Bim(J,I)$ (resp. $\SSBim(J,I)$) is naturally isomorphic to 
    a direct summand of $\res \circ \ind$.
    \item If $C\subset W_I\backslash W/W_J$ is downward-closed, then $\res$ and $\ind$ restrict to functors between $\SSBim_{\pi\inv{C}}(\emptyset,\emptyset)$ and $\SSBim_C(J,I)$.
\end{enumerate}
\end{lem}

\begin{proof}
By our assumption of generalized Demazure surjectivity, the inclusion $R^I \subset R$ is a Frobenius extension (see \cite[Section 4.1]{EKLP1}). Therefore the map $R^I \to R$ is split as an $(R^I,R^I)$-bimodule map (and the same for $R^J$).
Since $\res \circ \ind$ is given by the tensor product $\!_{R^I}R\otimes_{R^I}-\otimes_{R^J} R_{R^J}$, the natural inclusion of the functor $\!_{R^I}R^I\otimes_{R^I}-\otimes_{R^J} R^J_{R^J}$ is a direct summand. The first statement follows from the natural isomorphism between the latter functor and the identity functor on $\Bim(J,I)$.

Now fix $C$, a downward-closed subset in $W_I\backslash W/W_J$, and consider the second statement. That $\ind$ restricts follows from Proposition~\ref{ass}, since if $I_\bullet\expr p\in C$ is a reduced $(I,J)$-expression then 
\[ K_\bullet:=[[\emptyset \subset I]]\circ I_\bullet \circ [[J\supset \emptyset]]\expr \ma{p}\in \pi\inv{C}\] is reduced and thus $\ind(\BS(I_\bullet))\cong\BS(K_\bullet)$ belongs to $\SSBim_{\pi\inv{C}}(\emptyset,\emptyset)$.

We now show that $\res$ restricts. By Proposition~\ref{ass} again, it is enough to show that $\res(\BS(K_\bullet))$, where $K_\bullet\expr \{w\}$ is a reduced $(\emptyset,\emptyset)$-expression for $w\in \pi\inv{C}$, belongs to $\SSBim_C(J,I)$. 
By Proposition~\ref{ass}, the latter reduces to showing that every subordinate path of the expression 
\[I_\bullet:=[[I\supset \emptyset]]\circ K_\bullet\circ[[\emptyset\subset J]]\]
terminates in $C$.
But since 
\[I_\bullet \expr W_I w W_J=\pi(w)\]
we have by Theorem~\ref{newbruhatthm}
\[\Term(I_\bullet)= \{ \leq \pi(w) \}\subset C\]
as desired.
\end{proof}

As the last ingredient for Proposition~\ref{N=Gamma}, here is \cite[Lemma 4.14.(2)]{SingSb} which says that the support filtration from Definition~\ref{def.Gamma} is compatible with restriction.

\begin{lem}\label{ResGamma}
For $C\subset W_I\backslash W/W_J$ we have 
\begin{equation*}
\Gamma_C \circ \res = \res\circ \Gamma_{\pi\inv C}.
\end{equation*}
\end{lem}

\begin{proof}[Proof of Proposition \ref{N=Gamma}]
 By Lemma~\ref{lem:soequalonsummands} and Lemma~\ref{BinOrdinarySB}(1), we may restrict our attention to objects of the form $\res B\in \SSBim(J,I)$, where $B\in\SBSBim(\emptyset,\emptyset)$.

By \cite[Proposition 12.26]{GBM} and \cite[Theorem 10.32]{GBM}, the functors $\{N_D\}$, for downward-closed finite subsets $D\subset W$, induces a $\nabla$-filtration on $B\in\SBSBim(\emptyset,\emptyset)$. So does $\{\Gamma_D\}$ by Proposition~\ref{supportfiltr}, and $N_D \subset \Gamma_D$ by Lemma~\ref{NsubGamma}. But an inclusion of two functorial $\nabla$-filtrations must be an equality, since the associated graded of the filtrations have the same size. (To show equality for a given $D$, we can choose an enumeration of the Bruhat order compatible where $D = \{p_j \mid j \le i\}$ for some $i$.)

We thus get the second equality in
\[\Gamma_C(\res B)=\res(\Gamma_{\pi\inv C} B) = \res (N_{\pi\inv C}B)\subset N_C(\res B),\] 
where we use Lemma~\ref{ResGamma} in the first equality and Lemma~\ref{BinOrdinarySB}(2) in the final containment. Lemma~\ref{NsubGamma} then completes the proof.
\end{proof}

\subsection{Filtrations on morphisms and resolutions: the ordinary setting} \label{ss.nowformorphisms}

 In this section, we translate filtrations on ordinary Soergel bimodules to filtrations on morphism spaces, by reformulating a result of \cite{LibWil} on $\SBim=\SSBim(\emptyset,\emptyset)$. In the next section we deduce results about singular Soergel bimodules.

We use the standard notation for ordinary Soergel bimodules as in \cite{GBM,LibWil} etc. 
Following \cite[Definition 1]{LibWil}, we say that a complex $E_\bullet$ of $(R,R)$-bimodules is \emph{$\nabla$-exact} if each $E_j$ has $\nabla$-filtration, and for each $x\in W$ the induced complex $(\Gamma_{\leq x}/\Gamma_{< x})(E_\bullet)$ is exact (see also \cite[Remark 5(5)]{LibWil}).  

\begin{lem}\label{lem.tiltresol}
Let $C\subset W$ be finite and downward-closed and let $M$ be a graded $(R,R)$-bimodule with a $\nabla$-filtration. If $M=\Gamma_C M$, then there exists a finite $\nabla$-exact complex 
\begin{equation}\label{eq.tiltresol}
0\to B_k\to \cdots \to B_1\to B_0\to M\to 0    
\end{equation}
where $B_j\in \SBim_C$. 
\end{lem}
\begin{proof}
    We prove the statement by induction on $|C|$. Suppose the claim is true for $C\subset W$ and let $M=\Gamma_{C\sqcup x}M$ where $x\in W\setminus C$ and $C \sqcup x$ is downward-closed.
Then we have an exact sequence
\begin{equation}\label{eq.M}
    0\to \Gamma_C M\to M\xrightarrow{} R_x^{\oplus m}\to 0,
\end{equation}
for some $m\in\mathbb N[v,v\inv]$,
since $M$ is $\nabla$-filtered.
Here $R_x^{\oplus m}$ is shorthand for the shifted sum $\oplus_i R_x(i)^{\oplus m_i}$ where\footnote{Typical notation would have $m = \sum_i m_iv^{-i}$, an element of $\Z[v,v^{-1}]$.} $m_i\in\mathbb N$. By induction, $\Gamma_C M$ admits a resolution of the form in \eqref{eq.tiltresol}, and we let $B_i'$ be the various chain objects.

Let $\BS(x)$ abusively denote the Bott-Samelson bimodule for a reduced expression for $x$. Let
\begin{equation}\label{eq.rou}
0\to E_{\ell(x)}\to \cdots \to E_1\to E_0=\BS(x)\xrightarrow{g_x} R_x(\ell(x))    
\end{equation}
be the augmented Rouquier complex for $\BS(x)$ from \cite{LibWil}. That is, ignoring the last term, $E_{\bullet}$ is the tensor product of $[R(-1) \to B_s]$ as $s$ varies along the chosen reduced expression for $x$. Consequently, $E_i\in \SBim_C$ for $i>0$; see \cite[\S 2.3]{LibWil}.

Our construction of \eqref{eq.tiltresol} follows the analogous construction of a projective resolution of a short exact sequence in an abelian category, using the horseshoe lemma. Namely, $B_i:=B'_i\oplus E_i^{\oplus m}$ and the differentials are constructed inductively as explained below.

We have the solid part of the following commutative diagram of degree 0 morphisms in graded $(R,R)$-bimodules. The first row is \eqref{eq.M}, the other rows are split short exact sequences (with inclusion and projection maps), the left column $B'_\bullet$ is a $\nabla$-exact complex given by the induction hypothesis, and the right column is a direct sum of shifts of \eqref{eq.rou}.
    \begin{equation}
    \begin{tikzcd}\label{eq.cd}
      &0&0&0\\
        0\arrow[r] & \Gamma_C M\arrow[r]\arrow[u] &M\arrow[r,"\phi"]\arrow[u] &R_x^{\oplus m}\arrow[r]\arrow[u]& 0\\
        0\arrow[r] & B'_0 \arrow[r]\arrow[u,"g'"] &B'_0\oplus  \BS(x)^{\oplus m}\arrow[r]\arrow[u,dashed,"g"] &\BS(x)^{\oplus m}\arrow[r]\arrow[u,"g_x^{\oplus m}"]\arrow[ul,dotted,"f_{-1}"]& 0\\
        0\arrow[r] & B'_1 \arrow[r]\arrow[u,"d_{B'}"] &B'_1\oplus E_1^{\oplus m}\arrow[r]\arrow[u,dashed,"d_B"] &E_1^{\oplus m}\arrow[r]\arrow[u,"d_E"]\arrow[dotted]{ull}[near end]{f_{0}}& 0\\
        0\arrow[r]&B_2'
        \arrow[r]\arrow[u,"d_{B'}"] &B_2'\oplus E_2^{\oplus m}
        \arrow[r]\arrow[u,dashed,"d_B"] &E_2^{\oplus m}
        \arrow[u,"d_E"]\arrow[dotted]{ull}[near end]{f_{1}}\arrow[r]&0\\
        &\vdots\arrow[u,"d_{B'}"]&\vdots\arrow[u,dashed,"d_{B}"]&\vdots\arrow[u,"d_{E}"]
    \end{tikzcd}
    \end{equation}

    The first row is $\nabla$-exact by \cite[Remark 5(3)]{LibWil}; every other row is $\nabla$-exact since it is split and has $\nabla$-filtered terms;
    the right column is $\nabla$-exact by \cite[Corollary 3.9]{LibWil}.
    
    By \cite[Proposition 3.5]{LibWil}, the first row being $\nabla$-exact implies that $g_x^{\oplus m}$ lifts to some $f_{-1}$ as pictured. 

    The map $\bar h:=f_{-1}\circ d_E$ 
satisfies 
\[\phi\circ \bar h= \phi\circ f_{-1}\circ d_E=g_x^{\oplus m}\circ d_E = 0\] and thus restricts to a map $h\in \Hom^0(E_1^{\oplus m},\Gamma_CM)$. 
    By the $\nabla$-exactness of the first column
    \[g'\circ -:\Hom^0(E_1^{\oplus m},B_0')\to \Hom^0(E_1^{\oplus m},\Gamma_CM)\]
    is surjective. Thus there exists $f_0:E_1^{\oplus m}\rightarrow B_0$ with $h=g'\circ f_0$. 
    
    The maps $f_{i}:E_{i+1}^{\oplus m}\to B'_{i}$, for $i>0$ are constructed inductively as follows.
    The map $f_{i-1}\circ d_E$ is in the kernel of    
    \[d_{B'}\circ -:\Hom^0(E_{i+1}^{\oplus m},B'_{i})\xrightarrow{}\Hom^0(E_{i+1}^{\oplus m},B'_{i-1})\]
    and thus, by \cite[Proposition 3.5]{LibWil} and $\nabla$-exactness, has a lift which we call $f_i$, i.e., $d_{B'}\circ f_i= f_{i-1}\circ d_E$.
    
    Now we set $g=g'-f_{-1}$ and the other dashed maps $d_B=d_{B'}+d_E+(-1)^i f_i$. Then one can check that $d_B^2=0$. This completes the diagram \eqref{eq.cd}.
    
    Applying $\Gamma_{\leq y}/\Gamma_{< y}$, for each $y\in W$, to the diagram \eqref{eq.cd} we have a commutative diagram of the same form where the solid part remains exact. Thus the dashed part also remains exact and provides a desired complex \eqref{eq.tiltresol}.
\end{proof}

Recall that $\Hom_C(B,B')$ denotes the subspace of morphisms spanned by morphisms which factor through an object in $\SBim_C$.

\begin{lem}\label{lem.LW}
    Let $B,B'\in \SBim$ and let $C\subset W$ be a finite downward-closed set. A morphism $f:B\to B'$ has $\im f\subset \Gamma_C B'$ if and only if $f\in \Hom_C(B,B')$.
\end{lem}
\begin{proof}
If $f \in \Hom_C(B,B')$ then $\im f \subset N_C B'$ by definition. So $\im f \subset \Gamma_C B'$ by Proposition~\ref{N=Gamma}.

Now let $f$ be such that $\im f \subset \Gamma_C B'$. Let $f':B\to \Gamma_C B'$ be the restriction.
Since $B'$ has a $\nabla$-filtration, $\Gamma_C B'$ also has a $\nabla$-filtration. 
By Lemma~\ref{lem.tiltresol}, $\Gamma_C B'$ has a $\nabla$-exact resolution $E_\bullet$ in $\SBim_C$.
Now \cite[Proposition 3.5]{LibWil} says that $f'$ factors through the epimorphism $E_0\to \Gamma_C B'$. Thus $f$ also factors through $E_0\in\SBim_C$, that is, $f\in\Hom_C(B,B')$.  
\end{proof}

\subsection{Filtrations on morphisms and the one-tensor: the singular setting}\label{ss.application2}

\begin{thm}\label{lem.intermsBS}
Let $B,B'\in \SSBim$. Then we have
    \[\Hom_{<p}(B,B')=\Hom(B,\Gamma_{<p}B')\]
    where we view the right hand side as the submodule of $\Hom(B,B')$ consisting of the morphisms whose image is in $\Gamma_{<p}B'$.
\end{thm}

\begin{proof}
If $q<p$, the image of any morphism $B_q\rightarrow B'$ lies in $N_{<p}B'$ by definition. The inclusion $\Hom_{<p}(B,B')\subset\Hom(B,\Gamma_{<p}B')$ follows by Proposition~\ref{N=Gamma}. 

To show the other inclusion, we let $f:B\to \Gamma_{<p}B'$ and want to write $f$ as a sum of morphisms factoring through objects in $\SSBim_{<p}(J,I)$. By Lemma~\ref{BinOrdinarySB}, the identity functor on $\SSBim(J,I)$ is a summand of $\res \ind$, which implies that
\[ f = \proj \circ (\res \ind f) \circ \incl.\]
Here, $\proj$ and $\incl$ are the projection and inclusion for the identity functor inside $\res \ind$. If we can prove that $\res \ind f$ factors through $\SSBim_{<p}(J,I)$, then so will $f$.

Consider the morphism
\[\ind f :\ind B\to \ind \Gamma_{<p} B' = \Gamma_{\pi\inv\{<p\}}\ind B'.\]
By Lemma~\ref{lem.LW}, $\ind f = \sum g_j$ where each $g_j$ is a composition $\ind B\to B_{y_j}\to \ind B'$ for some $y_j\in \pi\inv\{<p\}$.
Since $\res$ is an additive functor, $\res \ind f = \sum \res g_j$, and $\res g_j$ factors through $\res B_{y_j}\in \SSBim_{<p}(J,I)$, proving the result. Several steps above are justified with  Lemma~\ref{BinOrdinarySB}.
\end{proof}

An important feature of $\nabla$-exact complexes is that, for any Soergel bimodule $B$, the functor $\Hom(B,-)$ is exact on them \cite[Proposition 3.5]{LibWil}. The analogous statement holds in the singular setting. We state it here only in a special case for simplicity.

\begin{lem}\label{lem.nablaexact}
    Let $B,B'\in \SSBim(J,I)$ and let $C\subset W_I\backslash W/W_J$ be downward-closed. Then 
    the short exact sequence 
    \begin{equation}\label{sesB}
        0 \rightarrow \Gamma_C B'\rightarrow B'\rightarrow B'/\Gamma_CB'\rightarrow 0
    \end{equation}
    induces an exact sequence
    \[ 0 \rightarrow \Hom(B,\Gamma_C B')\rightarrow \Hom(B,B')\rightarrow \Hom(B,B'/\Gamma_CB')\rightarrow 0.\]
\end{lem}
\begin{proof}
The functor $\Hom(B,-)$ is left-exact. By applying it on \eqref{sesB} we get an exact sequence
\[ 0 \rightarrow \Hom(B,\Gamma_C B')\rightarrow \Hom(B,B')\xrightarrow{r} \Hom(B,B'/\Gamma_CB').\]
It remains to show that $r$ is surjective.

Notice that all the terms in \eqref{sesB} have a $\nabla$-filtration and we have $\ch{B'}=\ch{\Gamma_CB'}+\ch{B'/\Gamma_CB'}$. So we can apply Soergel--Williamson's hom formula \cite[Theorem 7.9]{SingSb} to deduce that 
\[\gdim( \Hom(B,B'))=\gdim (\Hom(B,\Gamma_C B')+\gdim\Hom(B,B'/\Gamma_CB'),\] where $\gdim$ denotes the graded dimension of a vector space.
Then, by dimension reasons, $r$ must be surjective and the claim follows.    
\end{proof}


\begin{lem}\label{lem:ses}
Let $B\in \SSBim(J,I)$ and let $p$ a $(I,J)$-coset.  Then
    \[ \Hom(B,B_p)/\Hom_{ <p}(B,B_p)\cong\Hom(B,R_p(\ell(\ma{p})-\ell(J))).\]
\end{lem}

\begin{proof}
Recall by \Cref{lem.intermsBS} that 
$\Hom_{<p}(B,B_p)=\Hom(B,\Gamma_{<p}B_p)$ and that $B_p/\Gamma_{<p} B_p\cong \nabla_p = R_p(\ell(\ma{p})-\ell(J))$, see \eqref{eq:chbq}. Then by \Cref{lem.nablaexact} we obtain
the exact sequence
\[0 \rightarrow \Hom_{<p}(B,B_p)\rightarrow \Hom(B,B_p) \rightarrow \Hom(B,R_p(\ell(\ma{p})-\ell(J)))\rightarrow 0\]
and the statement follows.
\end{proof}

The reader may wish to review material on the one-tensor $\onetensor$, introduced in \S\ref{SSBimBackground}.

\begin{prop}\label{prop::kill1tensor}
Let $I_\bullet$ and $I'_\bullet$ be reduced expressions for a $(I,J)$-coset $p$. Then 
  \[\Hom^0_{<p}(\BS(I_\bullet),\BS(I'_\bullet))=\{ f\in  \Hom^0(\BS(I_\bullet),\BS(I'_\bullet))\mid f(\onetensor_{I_\bullet})=0\}.\]
\end{prop}

\begin{proof}
 Since both $I_\bullet\expr p$ and $I'_\bullet$ are reduced, both the elements $\onetensor_{I_\bullet}\in \BS(I_\bullet)$ and $\onetensor_{I'_\bullet}\in \BS(I'_\bullet)$ are in degree $-\ell(\ma{p})+\ell(J)$.
Thus if $f\in  \Hom^0(\BS(I_\bullet),\BS(I'_\bullet))$ we have
$f(\onetensor_{I_\bullet})=\lambda\onetensor_{I'_\bullet}$ for some $\lambda\in\Bbbk$.

Recall that $\nabla_p = R_p(\ell(\ma{p})-\ell(J)$, so $1 \in \nabla_p$ lives in degree $\ell(J) - \ell(\ma{p})$. Proposition~\ref{ass}, together with the definition of the character map, says that we have a short exact sequence of graded $(R^I,R^J)$-bimodules
\begin{equation}\label{eq.gammases}
0\to\Gamma_{<p}\BS(I'_\bullet)\to \BS(I'_\bullet)\xrightarrow{g} \nabla_p\to 0
\end{equation}
where $g$, after rescaling, sends $\onetensor_{I'_\bullet}$ to $1\in \nabla_p$.
Applying $\Hom^0(\BS(I_\bullet),-)$, we obtain by \Cref{lem.nablaexact} the exact sequence
\begin{equation}\label{eq.Homgammases}
0\to\Hom^0_{<p}(\BS(I_\bullet),\BS(I'_\bullet))\to\Hom^0(\BS(I_\bullet), \BS(I'_\bullet))\xrightarrow{g\circ-} \Hom^0(\BS(I_\bullet),\nabla_p)\to 0
\end{equation}
where the first map is the inclusion map. 
By \cite[Theorem 7.9]{SingSb} we have 
\[\Hom^0(\BS(I_\bullet),\nabla_p)\cong \Bbbk,\] where the isomorphism sends $h \in \Hom^0(\BS(I_{\bullet}),\nabla_p)$ to the coefficient of $1$ in $h(\onetensor_{I_\bullet})$. This in turn yields the exact sequence
\begin{equation}\label{eq.Homgammases2}
0\to\Hom^0_{<p}(\BS(I_\bullet),\BS(I'_\bullet))\to\Hom^0(\BS(I_\bullet), \BS(I'_\bullet))\xrightarrow{\phi} \Bbbk\to 0.
\end{equation}
If $f \in \Hom^0(\BS(I_\bullet)$ then $g(f(\onetensor_{I_\bullet}))=g(\lambda \onetensor_{I'_\bullet})=\lambda \cdot 1$, with $\lambda$ as above. Thus $\phi(f) = \lambda$.
It follows that $f\in \Hom^0_{<p}(\BS(I_\bullet),\BS(I'_\bullet))=\operatorname{Ker}\phi$ if and only if $\lambda=0$ if and only if $f(\onetensor_{I_\bullet}) = 0$.
\end{proof}

 In the next proposition, \emph{restriction} refers to restriction from left $R^{\leftred(p)}$-modules to left $R^I$-modules. Any module so restricted still retains an action of $R^{\leftred(p)}$.

\begin{prop}
Let $I_\bullet$ and $I'_\bullet$ be reduced expressions for a $(I,J)$-coset $p$. Assume further that $I'_\bullet$ is of the form
\[I'_\bullet =[[I\supset \leftred(p)]]\circ K_\bullet.\] Then $\BS(I'_{\bullet})$ is restricted from $\BS(K_{\bullet})$ and admits an action of $R^{\leftred(p)}$. 
If $f\in \Hom_{<p}(\BS(I_\bullet),\BS(I'_\bullet))$, then  $\im f\cap (R^{\leftred(p)} \cdot \onetensor_{I'_\bullet})= 0$. 
\end{prop}

\begin{proof}
Letting $z\expr K_\bullet$, the $(\leftred(p),J)$-coset $z$ has $\mi{z}=\mi{p}$ and the left redundancy $\leftred(z)=\leftred(p)$.
It follows that $R_z$ also restricts to $R_p$, and thus the exact sequence
\begin{equation}
    0\to \Gamma_{<z}\BS(K_\bullet)\to \BS(K_\bullet)\xrightarrow{g} R_z\to 0
\end{equation}
restricts to 
\begin{equation}
    0\to \Gamma_{<p}\BS(I'_\bullet)\to \BS(I'_\bullet)\xrightarrow{g'} R_p\to 0.
\end{equation}
If $c \in R^{\leftred(p)}$ then $g'(c\cdot \onetensor_{I'_\bullet})=g(c\cdot \onetensor_{K_\bullet})=c \cdot 1$. 

Now for $f\in \Hom_{<p}(\BS(I_\bullet),\BS(I'_\bullet))$, we have $\im f\subset \Gamma_{<p}\BS(I'_\bullet)$ by \Cref{lem.intermsBS}. 
Thus if $c \cdot \onetensor_{I'_\bullet}\in \im f\cap (R^{\leftred(p)} \cdot \onetensor_{I'_\bullet})$, then     $g'(c\cdot \onetensor_{I'_\bullet}) = 0$ so $c = 0$. This proves the claim.
\end{proof}


\printbibliography
\end{document}